\documentclass[english]{article}
\usepackage[latin1]{inputenc}
\usepackage{amsmath,amsthm,amssymb,babel}

\newtheorem{theorem}{Theorem}
\newtheorem{definition}{Definition}
\newtheorem{proposition}{Proposition}
\newtheorem{corollary}{Corollary}

\def\neweq#1{\begin{equation}\label{#1}}
\def\endeq{\end{equation}}

\def\ep{\varepsilon}

\def\phi{\varphi}
\def\RR{{\mathbb R} }
\def\NN{{\mathbb N} }

\def\di{\displaystyle}
\def\ri{\rightarrow}

\date{}
\title{\sc {On a $p(\cdot)$-biharmonic problem with no-flux boundary condition}}
\author{\sc Maria-Magdalena Boureanu$\,^a$
, Vicen\c{t}iu R\u{a}dulescu$\,^{b}$ and Du\v{s}an Repov\v{s}$\,^{c}$\\
\small $^a\,$Department of Applied Mathematics, University of Craiova,
200585 Craiova,
Romania\\ \small $^b\,$Department of Mathematics, University of Craiova,
200585 Craiova,
Romania\\
\small $^c\,$Faculty of Mathematics and
Physics, University of Ljubljana,\\ \small Jadranska  19,  P. O. Box 2964, 1001 Ljubljana, Slovenia\\
\small E-mail: {\tt  mmboureanu@yahoo.com,}
  {\tt vicentiu.radulescu@math.cnrs.fr,} {\tt dusan.repovs@guest.arnes.si}}

\begin{document}
\baselineskip16pt \maketitle

\noindent{\small{\sc Abstract}. The study of fourth order partial differential equations has flourished in the last years, however, a $p(\cdot)$-biharmonic problem with no-flux boundary condition has never been considered before, not even for constant $p$. This is an important step further, since surfaces that are impermeable to some contaminants are appearing quite often in nature, hence the significance of such boundary condition. By relying on several variational arguments, we obtain the existence and the multiplicity of weak solutions to our problem. We point out that, although we use a mountain pass type theorem in order to establish the multiplicity result, we do not impose an Ambrosetti-Rabinowitz type condition, nor a symmetry condition, on our nonlinearity $f$.\\

\small{\bf 2010 Mathematics
Subject Classification:}  46E35, 35J60, 35J30, 35J35, 35J40, 35D30. \\
\small{\bf Key words:} variable exponent, new variable exponent subspace, $p(\cdot)$-biharmonic operator, nonlinear elliptic problems,  weak solutions, existence, multiplicity.}

\section{Introduction}

Fourth order PDEs have various applications, to micro-electro-mechanical systems, phase field models of
multiphase systems, thin film theory, thin plate theory, surface diffusion on solids,
interface dynamics, flow in Hele-Shaw cells,  see for example \cite{Danet, Fe, M.T.G}. Therefore many authors focused on the study of such problems with constant exponents, like Molica Bisci and Repov\v{s} \cite{Mo-Re}, Candito and Molica Bisci \cite{Ca-Mo}, or Liu and Squassina \cite{Li-Sq} etc. At the same time, many applications are generated by the elliptic problems with variable exponents, which have a large range of applications, due to electrorheological fluids
\cite{repovs2, ana1, hal, Sob6, rajruz, repovs1, Ru, Sob5, ana2}, thermorheological fluids \cite{AnRo}, elastic materials \cite{Zhikov1, ANS}, image restoration \cite{chen}, mathematical biology \cite{frag}, dielectric breakdown and electrical
resistivity \cite{BM}, polycrystal plasticity \cite{BMP} and sandpile growth \cite{BMPLR}. At the interplay of these two research directions, a natural interest goes to the $p(\cdot)$-biharmonic problems. This trend is quite fresh, starting probably in 2009, with the papers \cite{AElA1, Amr}, where the authors considered problems with the Navier boundary condition
\begin{equation}\label{Navier}
 u=\Delta u=0  \quad \text{on } \partial\Omega.
\end{equation}
The line of investigation was continued by several authors, see \cite{afrouzi, AElAO, AElA2, aou, Lin, kong1, kong2, YL}. Notice that all these studies focus on problems with the Navier boundary condition \eqref{Navier} and only one of them, \cite{aou}, also considers the Neumann type boundary condition
$$\frac{\partial u}{\partial \nu}=\frac{\partial}{\partial \nu}(|\Delta u|^{p(x)-2}\Delta u)=0\quad \mbox{on } \partial\Omega.$$
But if we think at the applicability to real-life situations, when the surfaces are impermeable to some contaminants, we are drawn to the no-flux boundary problems. Hence, inspired by the previous studies \cite{CPAA_mc, NARWA}, where second order problems with no-flux boundary conditions are treated in the framework of the variable exponent spaces, we propose the following problem.

\begin{equation}\label{newPr1}
\left\{\begin{array}{lll}
\Delta(|\Delta u|^{p(x)-2}\Delta u )+a(x)|u|^{p(x)-2}u=\lambda\, f(x,u)
&\mbox{for}& x\in\Omega,\\
u\equiv \mbox{constant},\,\, \Delta u=0 &\mbox{for}& x\in\partial\Omega,\\
\displaystyle\int_{\partial\Omega}\frac{\partial}{\partial \nu}(|\Delta u|^{p(x)-2}\Delta u)\; dS=0,
\end{array}\right.
\end{equation}
where $\Omega\subset\RR^N$ ($N\geq 2$) is a bounded domain with
sufficiently smooth boundary, $\lambda >0$, and the exponent $p$ is log-H\"{o}lder continuous, that is,  for each $i\in\{1,\dots,N\}$ there exists $\bar{c}>0$ such that
\[ |p( x)-p( y)|\leq \frac{\bar{c}}{-\operatorname {log}|x- y|}\,
\ \mbox {\ \ for all\ }\ x,\, y\in\Omega,\ \ 0<| x- y|\leq
\frac{1}{2},\] and $$1<\di{\rm ess}\inf_{x\in\Omega}\di p(x)\leq \di{\rm ess}\sup_{x\in\Omega}\di p(x)<\infty \mbox {\ \ for all\ }\ x\in\Omega.$$
For simplicity,  we denote
$$h^-=\di{\rm ess}\inf_{x\in\Omega}\di h(x)\qquad\mbox{and }h^+=\di{\rm ess}\sup_{x\in\Omega}\di h(x).$$
We will work under the following hypotheses:
 \begin{description}
 \item[(H1)] $a\in L^\infty(\Omega)$ and there exists $a_0>0$ such that $a(x)\geq a_0$ for all $x\in\Omega$;

 \item[(H2)] $f:\Omega\times \mathbb{R} \rightarrow\mathbb{R}$ is a  Carath\'eodory function and there exist $t_0>0$ and a ball $B$ with $\overline{B}\subset\Omega$ such that $$\int_{B} F(x, t_0)\,dx>0,$$ where $F$ represents the antiderivative of $f$, that is, $F(x,t)=\int_{0}^{t}f(x,s)\,ds;$

 \item[(H3)] $\di\lim_{|t|\rightarrow\infty}\di\frac{f(x,t)}{|t|^{p(x)-1}}=0$ uniformly with respect to $x\in\Omega$;

 \item[(H4)] $\di\lim_{|t|\rightarrow 0}\di\frac{f(x,t)}{|t|^{p(x)-1}}=0$ uniformly with respect to $x\in\Omega$.
\end{description}
Note that all the necessary details regarding the definition and the properties of the variable exponent spaces involved in the investigation of our problem will be provided in the next section. It is worth mentioning though, that, since the class of problems represented by \eqref{newPr1} was not introduced before, not even for the constant case, we will need to introduce a new space on which is more appropriate to search for weak solutions to \eqref{newPr1}. Depending on the values taken by $\lambda$, we establish an existence and a multiplicity result. For the existence result we rely on a classical theorem from the field of calculus of variations, sometimes referred to as a Weierstrass-type theorem. For the second solution, we make use of a mountain pass type theorem, without imposing the usual Ambrosetti-Rabinowitz growth condition, that is, there exist $\theta>p^+$ and $l>0$ such that
$$0<\theta F(x,t)\leq f(x,t)t\quad\mbox{for all }|t|>l\mbox{ and a.e. }x\in\!{\Omega}.$$
The celebrated mountain pass theorem of Ambrosetti and Rabinowitz has provided lots of applications during the years and represents the key ingredient to the weak solvability for numerous problems. However, for the multiplicity of solutions, all the adaptations of the mountain pass theorem are relying on additional symmetry conditions on the nonlinearity $f$:
\begin{equation}\label{eq:symm}
f(x,-t)=-f(x,t)\quad\mbox{for a.e. }x\in\Omega\mbox{ and all }t\in\RR
\end{equation}
with the help of which we can get the existence of an unbounded sequence of weak solutions. This was the case for the fourth order PDEs with variable exponent treated by \cite{afrouzi, AElA2, AElA1, Amr, Lin}. Other multiplicity results, which do not impose condition \eqref{eq:symm} on the nonlinearity, were provided due to various three critical points theorems of Ricceri type, see \cite{AElAO, aou}. Our problem is the first variable exponent problem of fourth order for which the multiplicity of solutions is obtained by applying a different strategy. For second order problems with variable exponents for which the same strategy is applied we refer to \cite{CPAA_mc, NARWA, hosim}.

\section{Some preliminaries}
We introduce some notation that will clarify what follows. Thus, when we refer to a Banach space $X$, we denote by $X^\star$ its dual and by $\langle\cdot,\cdot\rangle$ the duality pairing between $X^\star$ and $X$. By $|\cdot|$ we denote the absolute value of a number, or the Euclidean norm when it is defined on $\RR^N$ ($N\geq 2$), respectively the Lebesgue measure, when it is applied to a set.

\smallskip

We recall the definitions of the variable exponent Lebesgue and Sobolev spaces and some of their basic properties, but much more details can be found in the comprehensive works \cite{book_p_x, CUF, RadRep}. As stated from the beginning, everywhere below we consider $p$ to be log-H\"{o}lder continuous with $1<p^-\leq p^+<\infty$.

The Lebesgue space with variable exponent is defined by
$$L^{p(\cdot)}({\Omega})=\{u\,:\quad u \mbox{ is a
 measurable real--valued function such that }
\int_{\Omega}|u(x)|^{p(x)}\;dx<\infty\}.$$ This space is equipped with the
 {Luxemburg norm},
 \begin{equation*}\label{elnorm}
\|u\|_{L^{p(\cdot)}({\Omega})}=\inf\left\{\mu>0\;:\quad\;\int_{\Omega}\left|
\frac{u(x)}{\mu}\right|^{p(x)}\;dx\leq 1\right\},
\end{equation*}
  and it is a separable and reflexive Banach
space, see \cite[Theorem 2.5, Corollary 2.7]{KR}. Also, we have the following continuous embedding result.

\begin{theorem}(\cite[Theorem 2.8]{KR})\label{teor Leb inclusion} If $0 <|{\Omega}|<\infty$ and
$p_1$, $p_2\in C(\overline\Omega;\RR)$, $1\leq
p_i^-\leq p_i^+<\infty$ ($i=1,2$), are such that $p_1 \leq p_2$
in ${\Omega}$, then the embedding $L^{p_2(\cdot)}({\Omega})\hookrightarrow L^{p_1(\cdot)}({\Omega})$ is continuous.
\end{theorem}

The {$p(\cdot)$-modular} of the
$L^{p(\cdot)}({\Omega})$ space is represented by $\rho_{p(\cdot)}:L^{p(\cdot)}({\Omega})\rightarrow\RR$,
$$\rho_{p(\cdot)}(u)=\int_{\Omega}|u(x)|^{p(x)}\;dx,$$
  and we have some useful properties connecting this application to the Luxemburg norm, see for example \cite[Theorem 1.3, Theorem 1.4]{fanz}. If $u\in L^{p(\cdot)}({\Omega})$,
 then
\begin{equation}\label{L40}
\|u\|_{L^{p(\cdot)}({\Omega})}<1\;(=1;\,>1)\;\;\;\Leftrightarrow\;\;\;\rho_{p(\cdot)}(u) <1\;(=1;\,>1);
\end{equation}
\begin{equation}\label{L4}
\|u\|_{L^{p(\cdot)}({\Omega})}>1\;\;\;\Rightarrow\;\;\;\|u\|_{L^{p(\cdot)}({\Omega})}^{p^-}\leq\rho_{p(\cdot)}(u)
\leq\|u\|_{L^{p(\cdot)}({\Omega})}^{p^+};
\end{equation}
\begin{equation}\label{L5}
\|u\|_{L^{p(\cdot)}({\Omega})}<1\;\;\;\Rightarrow\;\;\;\|u\|_{L^{p(\cdot)}({\Omega})}^{p^+}\leq \rho_{p(\cdot)}(u)\leq\|u\|_{L^{p(\cdot)}({\Omega})}^{p^-};
\end{equation}
\begin{equation}\label{L06}
\|u\|_{L^{p(\cdot)}({\Omega})}\rightarrow
0\;(\rightarrow\infty)\;\;\;\Leftrightarrow\;\;\;\rho_{p(\cdot)} (u)\rightarrow 0\;(\rightarrow\infty).
\end{equation}
If, in addition, $(u_n)_n\subset L^{p(\cdot)}({\Omega})$, then
\begin{equation*}\label{L6}
\lim_{n\ri\infty}\|u_n-u\|_{L^{p(\cdot)}({\Omega})}=
0\;\;\;\Leftrightarrow\;\;\;\lim_{n\ri\infty}\rho_{p(\cdot)} (u_n-u)= 0\;\;\;\Leftrightarrow
\end{equation*}
\begin{equation}\label{L07}
\Leftrightarrow \;\;\;(u_n)_n \mbox{ converges to } u \mbox{ in measure and }\lim_{n\ri\infty}\rho_{p(\cdot)}
(u_n)=\rho_{p(\cdot)}(u).
\end{equation}

In addition, we benefit from a H\"older type
inequality:
\begin{equation}\label{Hol}
\left|\int_{\Omega} u(x)v(x)\;dx\right|
\leq2\,\|u\|_{L^{p(\cdot)}({\Omega})}\|v\|_{L^{p'(\cdot)}({\Omega})},
\end{equation}
for all $u\in L^{p(\cdot)}({\Omega})$ and $v\in
L^{p'(\cdot)}({\Omega})$ (see \cite[Theorem 2.1]{KR}), where we
denoted by $L^{p'(\cdot)}({\Omega})$ the dual of
$L^{p(\cdot)}({\Omega})$, obtained by conjugating the exponent
pointwise, that is,  $1/p(x)+1/p'(x)=1$, see \cite[Corollary
2.7]{KR}.

\smallskip

Passing to the definition of the Sobolev space with variable exponent, $W^{k,p(\cdot)}({\Omega})$, we set
\[
W^{k,p(\cdot)}(\Omega)=\{ u\in L^{p(\cdot)}(\Omega):
D^{\alpha}u\in L^{p(\cdot)}(\Omega), |\alpha|\leq k\},
\]
where $D^{\alpha}u=\frac{\partial^{|\alpha|}}{\partial x_{1}^{\alpha_1}
\partial x_{2}^{\alpha_2}\dots \partial x_{N}^{\alpha_N}}u$, with
$\alpha=(\alpha_1,\dots ,\alpha_N)$ is a multi-index and
$|\alpha|=\sum_{i=1}^N \alpha_i$. The space $W^{k,p(\cdot)}(\Omega)$ endowed
with the norm
\[
\|u\|_{W^{k,p(x)}(\Omega)}=\sum_{|\alpha|\leq k}\|D^{\alpha}u\|_{L^{p(\cdot)}(\Omega)},
\]
is a separable and reflexive Banach space too, see \cite[Theorem 3.1]{KR}. \\

The log-H\"{o}lder continuity of the exponent $p$ plays a decisive role in the following density results.
\begin{theorem}\label{th dens} (see \cite[Theorem 3.7]{DD} and \cite[Section 6.5.3]{CUF})
Assume that
$\Omega\subset\mathbb R^N$ ($N\geq 2$) is a bounded domain with
Lipschitz boundary and $p$ is log-H\"{o}lder continuous with $1<p^-\leq p^+<\infty$. Then $C^\infty(\overline{\Omega})$ is dense in $W^{k,p(\cdot)}({\Omega})$.
\end{theorem}
Notice that the functions from $C^{0,\mu}(\Omega)$ are log-H\"{o}lder continuous. Also, it is important to mention that although the log-H\"{o}lder continuity of the exponent is a sufficient condition for the above density result, it is not always necessary, see \cite{CUF, Zhikov_dens}.

Moreover, the following embedding theorem takes place.
\begin{theorem}(see \cite[Theorem 2.3]{fanz} and \cite[Section 6]{CUF})\label{th emb}
Let us consider $q\in C(\overline{\Omega};\RR)$ such that $1<q^-\leq q^+<\infty$ and $q(x)\leq p^{*}_{k}(x)$ for
 all $x\in \overline{\Omega}$, where
 \[
p_{k}^{*}(x)=\begin{cases}
\frac{Np(x)}{N-kp(x)}     & \text{if } kp(x)<N,\\
+\infty       & \text{if } kp(x)\geq N
\end{cases}
\]
for any $x\in \overline{\Omega}$, $k\geq1$.\\
Then there is a continuous embedding
\begin{align*}
W^{k,p(\cdot)}(\Omega)\hookrightarrow L^{q(\cdot)}(\Omega).
\end{align*}
If we replace $\leq$ with $<$ the embedding is compact.
\end{theorem}

\smallskip

Let us denote by $W_0^{k,p(\cdot)}(\Omega)$ the closure of $C_{0}^{\infty}(\Omega)$
in $W^{k,p(\cdot)}(\Omega)$. In fact, we are interested in the properties of the spaces $W^{2,p(\cdot)}(\Omega)$, $W_0^{1,p(\cdot)}(\Omega)$ and $W^{2,p(\cdot)}(\Omega)\cap W_0^{1,p(\cdot)}(\Omega)$. Due to the log-H\"{o}lder continuity of the exponent $p$, the space $W_0^{1,p(\cdot)}(\Omega)$ coincides with
$$W_0^{1,p(\cdot)}(\Omega)=\left\{u\in W^{1,p(\cdot)}(\Omega):\quad u=0\,\,{\rm on}\,\partial\Omega\right\},$$
and it can be endowed with the norm
$$\| u\|_{W_0^{1,p(\cdot)}(\Omega)}=\|\nabla u\|_{L^{p(\cdot)}({\Omega})},$$
due to the following Poincar\'{e} type inequality (see \cite[Proposition 2.3]{fan1}):
\begin{equation}\label{eq Poincare}
\| u\|_{L^{p(\cdot)}({\Omega})}\leq C \|\nabla u\|_{L^{p(\cdot)}({\Omega})}\quad\mbox{for all } u\in W_0^{1,p(\cdot)}(\Omega),
\end{equation}
where $C$ is a positive constant. The space $\left(W_0^{1,p(\cdot)}(\Omega), \|\cdot\|_{W_0^{1,p(\cdot)}(\Omega)}\right)$ is a separable and reflexive Banach space (see \cite[Proposition 2.1]{fan1}).

\smallskip

Obviously, the choice of the norms has a major influence on the development of the argumentation. Generally, we know that if $(X,\, \|\cdot\|_{X})$ and $(Y,\, \|\cdot\|_{Y})$ are Banach spaces, then $(X\cap Y,\, \|\cdot\|_{X\cap Y})$ is a Banach space too, where $\|u\|_{X\cap Y}=\|u\|_{X}+\|u\|_{Y}$. In our case, we have,
$$\|u\|_{W^{2,p(\cdot)}(\Omega)\cap W_0^{1,p(\cdot)}(\Omega)}=\|u\|_{W^{2,p(\cdot)}(\Omega)}+\|u\|_{W_0^{1,p(\cdot)}(\Omega)}=\|u\|_{L^{p(\cdot)}({\Omega})}+\|\nabla u\|_{L^{p(\cdot)}({\Omega})}+\sum_{|\alpha|=2}\|D^{\alpha}u\|_{L^{p(\cdot)}(\Omega)}.$$
Furthermore, $\left(W^{2,p(\cdot)}(\Omega)\cap W_0^{1,p(\cdot)}(\Omega),\,\|\cdot\|_{W^{2,p(\cdot)}(\Omega)\cap W_0^{1,p(\cdot)}(\Omega)}\right)$ is a separable and reflexive Banach space. In addition, we know that $\|\cdot\|_{W^{2,p(\cdot)}(\Omega)\cap W_0^{1,p(\cdot)}(\Omega)}$ and $\|\Delta(\cdot)\|_{L^{p(\cdot)}(\Omega)}$ are equivalent norms on $W^{2,p(\cdot)}(\Omega)\cap W_0^{1,p(\cdot)}(\Omega)$, see \cite[Theorem 4.4]{Fu}. However, taking into account the particularity of problem \eqref{newPr1}, which represents the subject of our investigation, the following representation of the norm might be best:
\begin{equation}\label{eq: norm a}
\|u\|_a=\inf \left\{\mu>0:\int_{\Omega}\left(\left|\frac{\Delta u(x)}{\mu}\right|^{p(x)}
+a(x)\left|\frac{u(x)}{\mu}\right|^{p(x)}\right)dx\leq 1\right\}
\end{equation}
for all $u\in W^{2,p(\cdot)}(\Omega)$ or $W^{2,p(\cdot)}(\Omega)\cap W_0^{1,p(\cdot)}(\Omega)$. The previously defined norm represents a norm on both $W^{2,p(\cdot)}(\Omega)$ or $W^{2,p(\cdot)}(\Omega)\cap W_0^{1,p(\cdot)}(\Omega)$ and it is equivalent to the usual norm defined here, see \cite[Remark 2.1]{aou}. Moreover, the modular inequalities that were appropriate for the norm of the Lebesgue space, can be extended to this situation, by proceeding similarly to \cite[Theorems 1.2-1.3]{fanz}. More precisely, for any $a$ taken as in (H1), we consider $\Lambda:W^{2,p(\cdot)}(\Omega)\ri\RR$ defined by
\begin{equation}\label{eq: def Lambda}
\Lambda(u)=\int_{\Omega }\left[ | \Delta u| ^{p(x)}+a(x) | u| ^{p(x)}\right] \,dx.
\end{equation}
Let us fix $u\in W^{2,p(\cdot)}(\Omega)\setminus\{0\}$. It is trivial to see that $\Lambda(\mu\,u)$ is even, and, for $\mu\in [0,\infty)$, $\Lambda(\mu\,u)$ increases strictly. Also,
let $\mu_n\ri\mu$. Since $(\mu_n)_n$ is bounded, $1<p^-\leq p^+<\infty$ and $a$ satisfies (H1), by Lebesgue's Dominated Convergence Theorem we deduce that $\Lambda(\mu_n\,u)\ri\Lambda(\mu\,u)$, hence $\Lambda(\mu\,u)$ is continuous.

Based on these properties of $\Lambda$, we have the following consequence.

\begin{corollary}\label{cor1 Lambda}
Let $u\in W^{2,p(\cdot)}(\Omega)\setminus\{0\}$. Then $\|u\|_a=|\kappa|$ if and only if $\Lambda\left(\frac u \kappa\right)=1$.
\end{corollary}

\begin{proof}
Without loss of generality, we can suppose that $\kappa>0$ because $\Lambda(\mu\,u)$ is even.

\smallskip
To show the direct implication, we consider that $\|u\|_a=\kappa$. Note that $\Lambda(0\cdot u)=0$ and that, whenever $\mu\ri\infty$, $\Lambda(\mu\,u)\ri\infty$. Therefore the continuity of $\Lambda(\mu\,u)$ ensures the existence of a $\mu_0\in (0,\infty)$ with the property that
\begin{equation}\label{eq: Lambda e 1}
\Lambda\left(\frac u \mu_0\right)=1.
\end{equation}

By \eqref{eq: norm a},
\begin{equation}\label{eq: norm a Lambda}
\kappa=\inf \left\{\mu>0:\Lambda\left(\frac u \mu\right)\leq 1\right\}.
\end{equation}
Since, by \eqref{eq: Lambda e 1}, $\mu_0\in \left\{\mu>0:\Lambda\left(\frac u \mu\right)\leq 1\right\}$, relation \eqref{eq: norm a Lambda} gives us
\begin{equation}\label{eq: kappa leq}
\kappa\leq\mu_0.
\end{equation}
At the same time,
$$\mu\geq\mu_0\quad\mbox{for all } \mu\in(0,\infty) \mbox{ such that }\Lambda\left(\frac u \mu\right)\leq 1$$
because $\Lambda\left(\frac u \mu_0\right)=1$ and $\Lambda(\mu\,u)$ increases strictly for $\mu\in [0,\infty)$. The previous inequality indicates that $\mu_0$ represents a lower bound for the set $\left\{\mu>0:\Lambda\left(\frac u \mu\right)\leq 1\right\}$, thus, by \eqref{eq: norm a Lambda},
\begin{equation}\label{eq: kappa geq}
\kappa\geq\mu_0.
\end{equation}
Putting together \eqref{eq: Lambda e 1}, \eqref{eq: kappa leq} and \eqref{eq: kappa geq}, we have obtained that $\|u\|_a=\kappa$ implies $\Lambda\left(\frac u \kappa\right)=1$.

\smallskip

For the reciprocal implication, let us assume that $\Lambda\left(\frac u \kappa\right)=1$. By proceeding as above, we first notice that $\kappa\in \left\{\mu>0:\Lambda\left(\frac u \mu\right)\leq 1\right\}$, hence, by \eqref{eq: norm a}, $\kappa\geq \|u\|_a$. Then, using again the monotonicity of $\Lambda(\mu\,u)$ for $\mu>0$, we deduce that $\mu_0$ represents a lower bound for the set $\left\{\mu>0:\Lambda\left(\frac u \mu\right)\leq 1\right\}$, so $\kappa\leq \|u\|_a$ and the conclusion follows.
\end{proof}

Now we are able to prove the modular-type inequalities that we previously mentioned.

\begin{proposition}
For $u,\, u_n\in W^{2,p(\cdot)}(\Omega)$ we have
\begin{equation}\label{eq:norm 1}
\| u\|_a <( =;>1) \Leftrightarrow \Lambda(
u) <( =;>1),
\end{equation}
\begin{equation}\label{eq:norm less}
\| u\|_a \leq 1\Rightarrow \| u\|_a
^{p^{+}}\leq \Lambda( u) \leq \| u\|_a ^{p^{-}},
\end{equation}
\begin{equation}\label{eq:norm greater}\| u\|_a \geq 1\Rightarrow \|
u\|_a ^{p^{-}}\leq \Lambda( u) \leq \| u\|_a
^{p^{+}},
\end{equation}
\begin{equation}\label{eq:norm conv}
\| u_n\|_a \to 0\,\,(\to \infty)\quad \Leftrightarrow \Lambda(
u_n) \to 0\,\,(\to \infty).
\end{equation}
\end{proposition}

\begin{proof}
For $\|u\|_a=0$, $\Lambda (u)=0$ and there is nothing to prove, thus we focus on the situation when $\|u\|_a\neq 0$. Let us denote $\|u\|_a=\kappa$. By Corollary \ref{cor1 Lambda} we have that $\Lambda\left(\frac u \kappa\right)=1$.

If $\kappa =1$, we immediately get that $\Lambda(u)=1$. Using again Corollary \ref{cor1 Lambda} we easily notice that the vice-versa holds too: if $\Lambda(u)=1$, then $\|u\|_a=1$.

If $\kappa <1$, the definition (\ref{eq: def Lambda}) enables us to write
$$\frac{1}{\kappa^{p^-}}\Lambda(u)\leq \Lambda\left(\frac u \kappa\right)\leq \frac{1}{\kappa^{p^+}}\Lambda(u).$$
Since $\Lambda\left(\frac u \kappa\right)=1$, we arrive at
$$\kappa^{p^+}\leq \Lambda(u)\leq \kappa^{p^-}<1.$$
Similarly, if $\kappa >1$,
$$1<\kappa^{p^-}\leq \Lambda(u)\leq \kappa^{p^+},$$
so the direct implication of \eqref{eq:norm 1} is proved, together with relations \eqref{eq:norm less} and \eqref{eq:norm greater}. Actually, the reciprocal implication of \eqref{eq:norm 1} is also true. Indeed, let us assume for example that $\Lambda(u)<1$. Then it is clear that $\|u\|_a<1$, otherwise, if $\|u\|_a\geq 1$, then, from what we have proved above, we get $\Lambda(u)\geq 1$, which contradicts our initial assumption. The case when $\Lambda(u)>1$ is similar.

Passing to the proof of \eqref{eq:norm conv}, if $\| u_n\|_a \to 0$, then \eqref{eq:norm less} implies
$$0\leq \Lambda(
u_n)\leq \| u_n\|_a ^{p^{-}}\ri 0,$$
while if $\| u_n\|_a \to \infty$, then \eqref{eq:norm greater} implies
$$\Lambda(
u_n)\geq \| u_n\|_a ^{p^{-}}\ri \infty.$$
Reciprocal, if $\Lambda(
u_n)\ri 0$, we use \eqref{eq:norm 1} and \eqref{eq:norm less} to arrive at $0\leq\| u_n\|_a \leq \left(\Lambda(
u_n)\right)^{1/p^{+}}\ri 0,$ while if $\Lambda(
u_n)\ri \infty$, we use \eqref{eq:norm 1} and \eqref{eq:norm greater} to arrive at
$\| u_n\|_a \geq \left(\Lambda(
u_n)\right)^{1/p^{+}}\ri \infty.$
\end{proof}

Since we are getting closer to our goal, that is, the discussion of problem \eqref{newPr1}, it is time to introduce the space where we will search for weak solutions to our problem and to establish some of its main properties.

\section{Weak solvability of the problem}
When treating a problem with no-flux boundary condition, we need to choose a variable exponent space that is more appropriate for our study than the ones presented in the previous section. Therefore we introduce the following subspace of $W^{2,p(\cdot)}(\Omega)$.

$$V=\left\{u\in W^{2,p(\cdot)}(\Omega):\quad u\big|_{\partial\Omega}\equiv{\rm constant}\right\}.$$
Notice that $V$ can be viewed also as
\begin{equation}\label{def 2 V}
V=\left\{u+c:\quad u\in W^{2,p(\cdot)}(\Omega)\cap W_0^{1,p(\cdot)}(\Omega),\, c \in\RR\right\}
\end{equation}
and we can prove the following result.

\begin{theorem}\label{t0: Banach space}
  $\left(V, \|\cdot\|_{W^{2,p(\cdot)}(\Omega)}\right)$  is a separable and reflexive Banach space.
\end{theorem}

\begin{proof}
Our goal is to prove that $V$ is a closed subspace of the separable and reflexive Banach space $\left(W^{2,p(\cdot)}(\Omega), \|\cdot\|_{W^{2,p(\cdot)}(\Omega)}\right)$. Let $(v_n)_n\subset V$ be such that it converges to $v\in W^{2,p(\cdot)}(\Omega)$. In order to prove our claim it is sufficient to show that $v\in V$.

Taking into account \eqref{def 2 V}, we are aware of the fact that there exist $(u_n)_n\subset W^{2,p(\cdot)}(\Omega)\cap W_0^{1,p(\cdot)}(\Omega)$ and $(c_n)_n\subset\RR$ such that, for all $n\in\NN$, $$v_n=u_n+c_n.$$ The equivalence of the norms $\|\cdot\|_{W^{2,p(\cdot)}(\Omega)\cap W_0^{1,p(\cdot)}(\Omega)}$ and $\|\Delta(\cdot)\|_{L^{p(\cdot)}(\Omega)}$ on $W^{2,p(\cdot)}(\Omega)\cap W_0^{1,p(\cdot)}(\Omega)$ enables us to write
$$
\|u_n-u_m\|_{W^{2,p(\cdot)}(\Omega)\cap W_0^{1,p(\cdot)}(\Omega)}\leq c \|\Delta(u_n-u_m)\|_{L^{p(\cdot)}(\Omega)} \leq
c \sum_{i=1}^N\left\|\frac {\partial^2}{\partial x_i^2}(u_n+c_n-u_m-c_m)\right\|_{L^{p(\cdot)}(\Omega)}
$$
where $c$ represents a generic positive constant that may vary along the calculus, as it is the case for the remaining of our paper.
Consequently,
\begin{equation}\label{eq sir cauchy}
\|u_n-u_m\|_{W^{2,p(\cdot)}(\Omega)\cap W_0^{1,p(\cdot)}(\Omega)}\leq c \|v_n-v_m\|_{W^{2,p(\cdot)}(\Omega)}.
\end{equation}
But $(v_n)_n$ is converging to $v$ in $\left(W^{2,p(\cdot)}(\Omega), \|\cdot\|_{W^{2,p(\cdot)}(\Omega)}\right)$, hence it is a Cauchy sequence, and \eqref{eq sir cauchy} implies that $(u_n)_n$ is a Cauchy sequence in the Banach space $\left(W^{2,p(\cdot)}(\Omega)\cap W_0^{1,p(\cdot)}(\Omega),\,\|\cdot\|_{W^{2,p(\cdot)}(\Omega)\cap W_0^{1,p(\cdot)}(\Omega)}\right)$. It follows immediately that $(u_n)_n$ is converging to a function $\overline{u}\in W^{2,p(\cdot)}(\Omega)\cap W_0^{1,p(\cdot)}(\Omega)$.

On the other hand, we have the continuous embedding $L^{p(\cdot)}(\Omega)\hookrightarrow L^{1}(\Omega)$, so
\begin{equation}\label{eq sir cauchy cn}
\|c_n-c_m\|_{L^{1}(\Omega)}\leq c \|c_n-c_m\|_{L^{p(\cdot)}(\Omega)}\leq c\|v_n-v_m
\|_{L^{p(\cdot)}(\Omega)}+ c\|u_m-u_n\|_{L^{p(\cdot)}(\Omega)}.
\end{equation}
Since both $(v_n)_n$ and $(u_n)_n$ are Cauchy sequences in $W^{2,p(\cdot)}(\Omega)$, respectively in  $W^{2,p(\cdot)}(\Omega)\cap W_0^{1,p(\cdot)}(\Omega)$, by the definition of the corresponding norms and by the boundedness of $\Omega$, we infer that $(c_n)_n$ is a Cauchy sequence in $(\RR, |\cdot|)$. Therefore $(c_n)_n$ is converging to a $\overline{c}\in\RR$ and we have obtained that $v=\overline{u}+\overline{c}\in V,$ which completes our proof.
\end{proof}

Now that we have established some basic properties of the space $V$, we are ready to introduce the definition of a weak solution to our problem. To this purpose, we consider a smooth function $u$ that verifies \eqref{newPr1} and, by applying Green's formula, we get
$$\int_{\Omega } | \Delta u| ^{p(x)-2}\Delta u \Delta v\, dx + \int_{\partial\Omega } \frac{\partial}{\partial \nu}\left(| \Delta u| ^{p(x)-2}\Delta u\right) v\, dx - \int_{\partial\Omega } | \Delta u| ^{p(x)-2}\Delta u \frac{\partial v}{\partial \nu}\, dx+ \int_{\Omega } a(x)| u| ^{p(x)-2}uv\, dx=$$
$$=\lambda\int_{\Omega }f(x,u)v\,dx\qquad\mbox{for all }v\in C^\infty(\Omega).$$
Taking into consideration the fact that $V$ is a closed subspace of $\left(W^{2,p(\cdot)}(\Omega), \|\cdot\|_{W^{2,p(\cdot)}(\Omega)}\right)$ together with the density result Theorem \ref{th dens} and the boundary conditions, we arrive at the following formulation.

\begin{definition}
We say that  $u\in V$   is a weak solution of the boundary value problem \eqref{newPr1}  if and only if

$$\int_{\Omega } | \Delta u| ^{p(x)-2}\Delta u \Delta v\, dx+ \int_{\Omega}a(x)|u|^{p(x)-2}uv\;dx-\lambda \int_{\Omega}f(x,u)v\;dx=0\quad\mbox{for all }v\in V.$$
\end{definition}

To be able to find a weak solution to \eqref{newPr1}, we rely on the critical point theory, thus to problem \eqref{newPr1} we associate the functional $$I:V\ri\RR, \qquad I=I_1-\lambda I_2,$$
 where
\begin{equation}\label{eq functionals}
I_1(u)=\int_{\Omega }\frac{1}{p(x)}\left[ | \Delta u| ^{p(x)}+a(x) | u| ^{p(x)}\right] \,dx\quad\mbox{and}\quad I_2(u)=\int_{\Omega}F(x,u)\,dx.
\end{equation}

 By proceeding similarly to \cite[Proposition 2.5]{aou}, one can establish the following.

 \begin{proposition}\label{Prop aou I_1}Let $I_1:V\ri\RR$ be  the above defined functional.
 \begin{itemize}

\item[(i)] $I_1$ is of class $C^1$, with the G\^{a}teaux derivative defined by
 $$\langle I'_1(u), v\rangle=\int_{\Omega } | \Delta u| ^{p(x)-2}\Delta u \Delta v\, dx+ \int_{\Omega}a(x)|u|^{p(x)-2}uv\;dx.$$

 \item[(ii)] $I_1$ is (sequentially) weakly lower semicontinuous, that is, for any $u\in V$ and any subsequence $(u_n)_n\subset V$
such that $u_n\rightharpoonup u$ weakly in $V$, there holds
$$\Phi(u)\leq\liminf_{n\ri\infty} \Phi(u_n).$$

 \item[(iii)] $I'_1:V\ri V^*$ is of type (S+), that is, $u_n\rightharpoonup u$ and $\limsup_{n\ri\infty}I'_1(u_n)(u_n-u)\leq 0$ imply that $u_n\ri u$.
     \end{itemize}
 \end{proposition}
Thus, due to the properties fulfilled by $f$, we can easily deduce that $I$ is of class $C^1$, with G\^{a}teaux derivative defined by
 $$\langle I'(u), v\rangle=\int_{\Omega } | \Delta u| ^{p(x)-2}\Delta u \Delta v\, dx+ \int_{\Omega}a(x)|u|^{p(x)-2}uv\;dx-\lambda \int_{\Omega}f(x,u)v\;dx$$
 so any critical point of $I$ is a weak solution to \eqref{newPr1}. Therefore, in what follows we focus on studying the existence and the multiplicity of the nontrivial critical points  of $I$.

\section{The existence result}
We base our first result on the classical theorem of calculus of variations:

\begin{theorem}{(see \cite[Section 2, Theorem 1.2]{Costa})}\label{tinf}
Assume that $X$ is a reflexive Banach space of norm $\|\cdot\|_X$
and the functional $\Phi:X\ri\RR$ is
\begin{itemize}

\item[(i)] coercive on $X$, that is, $\Phi(u)\ri\infty$ as
$\|u\|_X\ri\infty$;

\item[(ii)] (sequentially) weakly lower semicontinuous on
$X$.
\end{itemize}
Then $\Phi$ is bounded from below on $X$ and attains its infimum in
$X$.
\end{theorem}

By applying this result to the functional $I$, we prove the existence of a nontrivial weak solution.

\begin{theorem}\label{teorr1}
Let $\Omega\subset\RR^N$ ($N\geq 2$) be a bounded domain with
smooth boundary and $p$ be log-H\"{o}lder continuous with $1<p^-\leq p^+<\infty$ for all $x\in\Omega.$
Assume hypotheses (H1)-(H3) take place. Then there exists a constant $\lambda_0>0$ such that problem \eqref{newPr1} has at least one nontrivial weak solution in $V$ for every $\lambda>\lambda_0$.
\end{theorem}

\begin{proof}
Let us first deal with the coercivity of $I$.
Hypothesis (H3) implies that, for all $\varepsilon>0$, there exists $\delta_{\varepsilon}>0$ such that for all $|t|>\delta_{\varepsilon}$ and all $x \in \Omega$ we have
\begin{equation}\label{marginireJ2_30}\nonumber
 |f(x,t)|\leq\varepsilon|t|^{p^(x)-1}.
\end{equation}
For the moment, we arbitrarily fix $\varepsilon>0$. Then the continuity of $f$  in its second argument indicates that  for all $t \in \RR$ and all $x \in \Omega$ there exists $c_0>0$ such that
\begin{equation}\label{marginireJ2_32}
|f(x,t)|\leq c_0+\varepsilon|t|^{p^(x)-1}.
\end{equation}
Taking into account the definition of $I$ (see \eqref{eq functionals}) and \eqref{marginireJ2_32}, we arrive at
$$I(u)\geq \frac{1}{p^+}\int_{\Omega }\left[ | \Delta u| ^{p(x)}+(a(x)-\lambda\varepsilon) | u| ^{p(x)}\right]\, dx-\lambda\,c_0\|u\|_{L^{1}(\Omega)}.$$
Let us choose $\varepsilon$ such that $\varepsilon<a_0/\lambda$ because in this way
$$\widetilde{a}=a-\lambda\varepsilon$$
verifies (H1). We know that $V$ is endowed with the norm $\|\cdot\|_{W^{2,p(\cdot)}(\Omega)}$ which is equivalent to the norm $\|\cdot\|_{\widetilde{a}}$ introduced by \eqref{eq: norm a}. Then for any $u\in V$ with $\|u\|_{\widetilde{a}}\geq 1$, inequality \eqref{eq:norm greater} leads to
\begin{equation}\label{eq: I1 coerc}I(u)\geq \frac{1}{p^+}\|u\|_{\widetilde{a}}^{p^-}-\lambda\,c_0\|u\|_{L^{1}(\Omega)}.
\end{equation}
At the same time, we have that $1<p_2^\star(x)$ for all $x\in\overline{\Omega}$, therefore by Theorem \ref{th emb} and by \eqref{eq: I1 coerc} we deduce that  there exists $c>0$ such that
$$I(u)\geq \frac{1}{p^+}\|u\|_{\widetilde{a}}^{p^-}-\lambda\,c\|u\|_{\widetilde{a}},$$
hence $I$ is coercive.

\smallskip

Moving further, to the weakly lower semicontinuity of $I$, we already know that $I_1$ is weakly lower semicontinuous, by Proposition \ref{Prop aou I_1}. To investigate if this property holds for $I_2$ too,  we assume $u_n\rightharpoonup u$ in $V$.
But $V$ is a closed subspace of $W^{2, p(\cdot)}(\Omega)$ thus the compact embedding produced by Theorem \ref{th emb} gives us
 \begin{equation}\label{convlpm}
   u_n\rightarrow u \mbox{ in } L^{p^(\cdot)}(\Omega) \quad\mbox{and}\quad  u_n\rightarrow u \mbox{ in } L^{1}(\Omega).
\end{equation}

 Using the mean value theorem,there exists $v$ which takes values strictly between the values of $u$ and $u_n$ such that
 $$|I_2(u_n)-I_2(u)|\leq \int_{\Omega}|F(x,u_n)-F(x,u)|\,dx\leq \int_{\Omega}|u_n-u|\sup_{x\in\Omega}|f(x,v(x)|\,dx,$$
hence, by \eqref{marginireJ2_32} and \eqref{convlpm} the functional $I_{2}$ is weakly continuous, so $I$ is weakly continuous also. Consequently, we obtain the weakly lower semicontinuity of $I$.

\smallskip

Now we are in position to apply Theorem \ref{tinf} and to find $u_1\in V$ in which $I$ attains its infimum, hence $u_1$ represents a weak solution to problem \eqref{newPr1}. Furthermore, for all $\lambda>0$,
\begin{equation}\label{MINIM}
   I({u}_1)\leq I(u) \quad\mbox{for all}\quad u\in V.
\end{equation}
Given the ball $B$ provided by hypothesis (H2), we can take $\varepsilon>0$ sufficiently small such that
$$\overline{B_\varepsilon}:=\overline{\{x\in\Omega|\quad{\rm dist}(x,\,B)\leq\varepsilon\}}\subset\Omega.$$
Furthermore, we can construct the following $C_c^1$ function:
$$u_\varepsilon(x):=\left\{
                      \begin{array}{c}
                        t_0, \quad\mbox{when }\qquad x\in B, \\
                        0, \quad\mbox{when }x\in \Omega\setminus B_\varepsilon. \\
                      \end{array}
                    \right.
$$
Then
$$I(u_\ep)\leq I_1(u_\ep)-\lambda\int_{B}F(x,t_0)\,dx-\lambda \int_{B_\ep\setminus B}F(x,u_\ep)\,dx.$$
By the definition of $F$ we are able to fix $\ep_0$ sufficiently small such that there exists a positive constant $\alpha_0$ with the property that
$$I(u_{\ep_0})\leq I_1(u_{\ep_0})-\lambda\alpha_0\int_{B}F(x,t_0)\,dx.$$
Now, by taking
 \begin{equation}\label{eq:lambda0}
\lambda_0:=\frac{I_1(u_{\ep_0})}{\alpha_0\int_{B}F(x,t_0)\,dx}>0
\end{equation}
we deduce that
$I(u_{\ep_0})< 0$ for all $\lambda>\lambda_0$.
By choosing $u=u_{\ep_0}$ in \eqref{MINIM} we obtain that ${u}_1$ is nontrivial for all $\lambda>\lambda_0$ because $I(0)=0$, and we have completed our proof.
\end{proof}

\section{The multiplicity result}

For the multiplicity result of this paper we rely on a variant of the celebrated mountain pass theorem (see for example \cite{Jabri,KrisR,PRad}) of Ambrosetti and Rabinowitz.
\begin{theorem}\label{mpt} Let $\left(X,\|\cdot\|_{X}\right)$ be a Banach space. Assume that
$\Phi\in\!C^1(X;\RR)$ satisfies the Palais-Smale condition, that is, any sequence $(u_n)_n\subset X$ such that
$\left(\Phi(u_n)\right)_n$ is bounded and $\Phi'(u_n)\rightarrow 0$ in $X^{\star}$ as $n \rightarrow\infty$, contains a
convergent subsequence. Also, assume that $\Phi$ has a mountain pass geometry, that is,
\begin{itemize} \item[($i$)] there exist two constants $\tau>0$ and $\rho\in\RR$ such that $\Phi(u)\geq \rho$ if $\|u\|_X=\tau;$

\item[($ii$)] $\Phi(0)<\rho$ and there exists $e\in X$ such that $\|e\|_X>\tau$ and $\Phi(e)<\rho$.

\end{itemize} Then $\Phi$ has a critical point $u_0\in X\setminus \{0,\,e\}$ with
 critical value $$\Phi(u_0)=\di\inf_{\gamma\in {\cal P}}\sup_{u\in \gamma}\di \Phi(u)\geq{\rho}>0,$$
 where ${\cal P}$ denotes the class of the paths $\gamma\in C([0,1];X)$ joining 0 to $e$.
\end{theorem}

Now we are able to prove the following.

\begin{theorem}\label{teorr2}
Let $\Omega\subset\RR^N$ ($N\geq 2$) be a bounded domain with
smooth boundary and $p$ be log-H\"{o}lder continuous with $1<p^-\leq p^+<\infty$ for all $x\in\Omega.$
Assume hypotheses (H1)-(H4) take place. Then there exists $\lambda_0>0$ such that problem \eqref{newPr1} has at least two nontrivial weak solutions in $V$ for every $\lambda>\lambda_0$.
\end{theorem}

\begin{proof}
By Theorem \ref{teorr1} we have already established that problem \eqref{newPr1} has at least one nontrivial weak solution $u_1\in V$ for every $\lambda>\lambda_0$, where $\lambda_0$ is the one defined by \eqref{eq:lambda0}. To deduce the existence of a second nontrivial weak solution for problem \eqref{newPr1}, we will show that $I$ satisfies the hypotheses of Theorem \ref{mpt}. We begin with the Palais-Smale condition. Let us consider a sequence  $(u_n)_n\subset V$ with the property that there exits $M>0$ such that
\begin{eqnarray}\label{marginireIL3}
|I(u_n)|\leq M,\quad \mbox{and},\quad \mbox{when} \quad n\rightarrow\infty, \, \quad I'(u_n)\rightarrow0 \quad \mbox{in} \quad V^*.
\end{eqnarray}
We recall that in the proof of Theorem \ref{teorr1} we have established the coercivity of $I$, so by \eqref{marginireIL3} we infer that $(u_n)_n$ is bounded. Moreover,  $V$ is a reflexive Banach space and a closed subspace of $W^{2, p(\cdot)}(\Omega)$, thus there exists $u_0\in V \subset W^{2,p(\cdot)}(\Omega)$ such that, passing eventually to a subsequence,
\begin{eqnarray}\label{limitaUNinW}
u_n\rightharpoonup u_0 \quad \mbox{in} \quad W^{1,\overset{\rightarrow} p(\cdot)}(\Omega).
\end{eqnarray}
Applying again Theorem \ref{th emb} we deduce that
\begin{eqnarray}\label{limitaUNinWL}
u_{n} \rightarrow u_0 \quad \mbox{in}\quad L^1(\Omega), \qquad u_{n} \rightarrow u_0 \quad \mbox{in}\quad L^{p^-}(\Omega) \qquad \mbox{and}\qquad u_{n} \rightarrow u_0 \quad \mbox{in}\quad L^{p(\cdot)}(\Omega).
\end{eqnarray}
All the above information was obtained starting from the boundedness of $(I(u_n))_n$. By exploiting the second part of relation \eqref{marginireIL3} and the weak convergence from
 \eqref{limitaUNinW}  we arrive at
\begin{eqnarray}\label{limitaDerivataIn}\nonumber
\lim\limits_{n\rightarrow\infty}\left| \langle I'(u_n), u_n-u_0\rangle\right|=0.
\end{eqnarray}
More exactly, we have
\begin{eqnarray}\label{floare}
0 &=& \lim\limits_{n\rightarrow\infty}\int_{\Omega } | \Delta u_n| ^{p(x)-2}\,\Delta u_n \,\Delta (u_n-u_0)\, dx\nonumber\\
& &+\lim\limits_{n\rightarrow\infty}\int_{\Omega} a(x)|u_n|^{p(x)-2}u_n(u_n-u_0)\,dx\\
& &-\lim\limits_{n\rightarrow\infty}\lambda\int_{\Omega} f(x,u_n)(u_n-u_0)\,dx.\nonumber
\end{eqnarray}
By (H1), \eqref{Hol}, \eqref{limitaUNinWL} and \eqref{L07}, we deduce that
\begin{equation}\label{marginireBUn}
\lim_{n\rightarrow\infty}\left|\int_{\Omega} a(x)|u_n|^{p(x)-2}u_n(u_n-u_0)\,dx \right|\leq 2 \|a\|_{L^{\infty}(\Omega)} \lim_{n\rightarrow\infty}\left(\||u_n|^{p(x)-1}\|_{L^{p'(\cdot)}(\Omega)}\|u_n-u_0\|_{L^{p(\cdot)}(\Omega)}\right)=0.
\end{equation}
On the other hand, by \eqref{marginireJ2_32},  \eqref{Hol}, \eqref{limitaUNinWL} and \eqref{L07}, there exists $c>0$ such that
\begin{eqnarray}\label{steacu4}
\lim_{n\rightarrow\infty} \left|\int_{\Omega} f(x,u_n)(u_n-u_0)dx \right|&\leq& c_0\lim_{n\ri\infty}\|u_n-u_0\|_{L^1(\Omega)}\\
& & +c\,\lim_{n\ri\infty}\left(\||u_n|^{p^{-}-1}\|_{L^{{({p^{-}})^{'}}}(\Omega)}\|u_n-u_0\|_{L^{p^{-}}(\Omega)}\right)=0.\nonumber
\end{eqnarray}
Replacing \eqref{marginireBUn} and \eqref{steacu4} in \eqref{floare} we obtain
$$\lim\limits_{n\rightarrow\infty}\int_{\Omega } | \Delta u_n| ^{p(x)-2}\,\Delta u_n \,\Delta (u_n-u_0)\, dx=0,$$
so the weak convergence  \eqref{limitaUNinW} and Proposition \ref{Prop aou I_1} imply that $u_{n}\rightarrow u_0 $ {in} $V$ as $n\rightarrow\infty$. With this, we conclude that $I$ verifies the Palais-Smale condition. Let us show now that $I$ has a mountain pass-type geometry too.

\smallskip

We can see immediately that
\begin{equation}\label{eq: MPG I1 1}
I_1(u)\geq \frac{1}{p^+}\int_{\Omega }\left[ | \Delta u| ^{p(x)}+a(x) | u| ^{p(x)}\right] \,dx \quad\mbox{for all}\quad u\in V.
\end{equation}
Passing to $I_2$, we make use once again of (H3). For $\varepsilon>0$ arbitrarily fixed, there exists $\delta_1\geq 1$ such that for all $|s|>\delta_1$ and all $x \in \Omega,$
\begin{eqnarray}\label{limitaLaInfinitF1_1}\nonumber
 |f(x,s)|\leq\varepsilon |s|^{r-1}
\end{eqnarray}
where $p^{+}<r<{p}_2^\star$.
At the same time, by (H4), there exists $\delta_2>0$ such that for all $|s|<\delta_2$ and all $x \in \Omega,$
  \begin{eqnarray}\label{limitaLa0F4_1}\nonumber
 |f(x,s)|\leq\varepsilon |s|^{p(x)-1}.
\end{eqnarray}
Putting together the previous two inequalities and the continuity of $f$ in its second argument, for a sufficiently large constant $c>0$,
\begin{equation}\label{eq: MPG I2}
 I_2(u)\leq c\,\|u\|_{L^r(\Omega)}^{r}+\frac{\varepsilon}{p^{-}}\int_{\Omega }| u| ^{p(x)}\, dx\quad\mbox{for all}\quad u\in V.
\end{equation}
Combining \eqref{eq: MPG I1 1} and \eqref{eq: MPG I2} we get
$$I(u)\leq \left(\frac{1}{p^{+}}-\frac{\lambda\,\varepsilon}{a_0\,p^{-}}\right)\int_{\Omega }\left[ | \Delta u| ^{p(x)}+a(x) | u| ^{p(x)}\right] \,dx-\lambda \,c\,\|u\|_{L^r(\Omega)}^{r}\quad\mbox{for all}\quad u\in V,$$
since $a(x)\geq a_0>0$ for all $x\in\Omega$.

We arbitrarily take $0<\tau<1$. Then, due to the above inequality, for $\|u\|_a=\tau$, relation \eqref{eq:norm less} and Theorem \ref{th emb} give us
\begin{equation}\label{eq: MPG I1}
I(u)\geq \left(\frac{1}{p^{+}}-\frac{\lambda\,\varepsilon}{a_0\,p^{-}}\right)\|u\|_a^{p^+}-\lambda \,c\,\|u\|_{a}^{r}.
\end{equation}
We choose $0<\varepsilon<a_0\,p^-/({\lambda p^+})$ and, since $p^{+}<r$, for $\tau=\|u\|_{a}<\min\{1,\|{u}_1\|_{a}\},$ we can find $\rho$ such that $I(u)\geq \rho>0=I(0)>I(u_1)$, where $u_1$ is the first nontrivial weak solution found by Theorem \ref{teorr1}. Therefore $I$ has a mountain pass-type geometry.
 \smallskip

 Now we can apply Theorem \ref{mpt} to obtain a second nontrivial weak solution ${u}_2\in V\backslash \{0,\, {u}_1\}$ to problem (\ref{newPr1}) and our proof is complete.
\end{proof}

\bigskip
\textbf{Acknowledgment.} M.M. Boureanu was supported by a grant of Romanian National Authority for Scientific Research and Innovation, CNCS-UEFISCDI, project number PN-II-RU-TE-2014-4-0657. V. R\u adulescu was supported by a grant of Romanian National Authority for Scientific Research and Innovation, CNCS-UEFISCDI, project number PN-II-PT-PCCA-2013-4-0614. D. Repov\v s was supported in part by the Slovenian Research Agency grants P1-0292-0101, J1-6721-0101, J1-7025-0101 and J1-5435-0101.


\begin{thebibliography}{00}
{\footnotesize
\bibitem{afrouzi} G.A. Afrouzi, M. Mirzapour and N.T. Chung, Existence and non-existence of solutions for a $p(x)$-biharmonic problem, {Electronic Journal of Differential Equations},
 2015 (2015),  1--8.

\bibitem{AElAO} M. Allaoui, A. El Amrouss and A. Ourraoui, Three solutions for a quasi-linear elliptic problem, Applied Mathematics E-Notes, 13 (2013), 51--59.

\bibitem{AnRo} S.N. Antontsev and J.F. Rodrigues, On stationary thermorheological viscous
flows, {Ann. Univ. Ferrara Sez. VII Sci. Mat.}, {52} (2006),  19--36.

\bibitem{AElA2}  A. Ayoujil, A.R. El Amrouss, Continuous spectrum of a fourth order nonhomogeneous elliptic equation with variable exponent, Electron. J. Differential
Equations 2011 (24) (2011), 1--12.

\bibitem{AElA1} A. Ayoujil and A.R. El Amrouss, On the spectrum of a fourth order elliptic equation with variable exponent, Nonlinear Anal. 71 (2009), 4916--4926.

\bibitem{BM} M. Bocea and M. Mih\u ailescu, $\Gamma$ - convergence of power-law functionals with variable exponents, Nonlinear Anal. 73 (2010) 110--121.

\bibitem{BMP} M. Bocea, M. Mih\u ailescu and C. Popovici, On the asymptotic behavior of variable exponent power-law functionals and applications, Ricerche Mat. 59 (2010) 207--238.

\bibitem{BMPLR} M. Bocea, M. Mih\u ailescu, M. Perez-Llanos and J.D. Rossi, Models for growth of heterogeneous sandpiles via Mosco
convergence, Asymptotic Analysis 78 (2012), 11--36.


\bibitem{ANS} M.M. Boureanu,  A. Matei and  M. Sofonea, Nonlinear problems with $p(\cdot)$-growth conditions and applications to antiplane contact models, {Adv. Nonl. Studies} {14} (2014), 295--313.

\bibitem{CPAA_mc} {M.M Boureanu} and {C. Udrea}, No--flux boundary value problems with anisotropic variable exponents, {Commun. Pure Appl. Anal.}, {14} (2015), 881--896.

\bibitem{NARWA} {M.M. Boureanu} and {D.N. Udrea}, Existence and multiplicity results for elliptic problems with $p(\cdot)$ - growth conditions, {Nonl. Anal. RWA}, {14} (2013), 1829--1844.

\bibitem{Ca-Mo} P. Candito and G. Molica Bisci, Multiple solutions for a Navier boundary value problem
involving the $p$-biharmonic, Discr. Contin. Dyn. Syst. Ser. S, 5(4) (2012), 741--751.

\bibitem{repovs2} M. Cencelj, D. Repov\v s and Z. Virk, Multiple perturbations of a singular eigenvalue problem. Nonlinear Anal. 119 (2015), 37--45.

\bibitem{chen} Y. Chen, S. Levine, and R. Rao, Variable Exponent, Linear Growth Functionals in Image Restoration,
{SIAM J. Appl. Math.}, {66} (2006), 1383--1406.

\bibitem{Costa} D. G. Costa, {An Invitation to Variational Methods in Differential Equations}, Birkh\"{a}user
Boston, 2007.

\bibitem{CUF} D. Cruz-Uribe, A. Fiorenza, Variable Lebesgue Spaces: Foundations and Harmonic Analysis, Springer Basel, 2013.

\bibitem{Danet}  C.-P. D\u ane\c t, Two maximum principles for a nonlinear fourth order
equation from thin plate theory, Electronic J. Qualitative Theory of
Diff. Eq., 31 (2014), 1--9.

\bibitem{DD} L. Diening, {Maximal function on generalized Lebesgue spaces
$L^{p(\cdot)}$,} {Mathematical Inequalities and Applications},
{7} (2004), 245--253.

\bibitem{book_p_x} L. Diening,  P. Harjulehto,  P. H\"{a}st\"{o} and M. R\.{u}\v{z}i\v{c}ka, {Lebesgue and Sobolev Spaces with Variable Exponents}, Lecture Notes in Mathematics, Springer-Verlag Berlin Heidelberg, 2011.

\bibitem{Amr} A. El Amrouss, F. Moradi and M. Moussaoui,
{Existence of solutions for fourth-order PDEs with variable exponentsns},
{Electron. J. Differ. Equ.}, {2009} (2009),  1--13.

\bibitem{aou} A.R. El Amrouss and A. Ourraoui, Existence of solutions for a boundary problem
involving p(x)-biharmonic operator, Bol. Soc. Parana. Mat.,  31 (2013),  179--192.

\bibitem{fanz} X.L. Fan and D. Zhao,
{On the spaces $L^{p(x)}$ and $W^{m,p(x)}$},
{J. Math. Anal. Appl.}, {263} (2001), 424-446.

\bibitem{fan1} X.L. Fan,
{Solutions for $p(x)$-Laplacian Dirichlet problems with singular coefficients},
{J. Math. Anal. Appl.}, {312} (2005), 464--477.

\bibitem{Fe} A. Ferrero and G. Warnault,
{On a solutions of second and
fourth order elliptic with power type nonlinearities,} Nonlinear Anal.
T.M.A., (70)8 (2009), 2889--2902.

\bibitem{frag} G. Fragnelli, Positive periodic solutions for a system of anisotropic parabolic equations, J. Math. Anal. Appl. 367 (2010) 204--228.

\bibitem{ana1} Y. Fu and Y. Shan, On the removability of isolated singular points for elliptic equations involving variable exponent, Adv. Nonlinear Anal. 5 (2016), no. 2, 121--132.

\bibitem{hal} T.C. Halsey, Electrorheological fluids, {Science},
{258} (1992), 761--766.

\bibitem{hosim} K. Ho and I. Sim, Existence results for degenerate $p(x)$-Laplace equations with Leray-Lions type operators, Science China Mathematics, accepted.

\bibitem{Jabri} Y. Jabri, {{The Mountain Pass Theorem. Variants, Generalizations and Some Applications}}, Cambridge
    University Press, 2003.

\bibitem{kong1} L. Kong, Eigenvalues for a fourth order elliptic problem, Proc. Amer. Math. Soc., 143 (2015), 249--258.

\bibitem{kong2} L. Kong, On a fourth order elliptic problem with a $p(x)$-biharmonic operator, Appl. Math.
Lett., 27 (2014), 21--25.

\bibitem{KR} O. Kov\'a\v cik and J. R\'akosn\'{\i}k, On spaces
$L^{p(x)}$ and $W^{k,p(x)}$, {Czechoslovak Math. J.}, {41}
(1991), 592--618.

\bibitem{KrisR} A. Krist\'{a}ly, V. R\u{a}dulescu and C. Varga, {{Variational Principles in Mathematical Physics, Geometry,  and
Economics: Qualitative Analysis of Nonlinear Equations and Unilateral Problems}}, Encyclopedia of Mathematics and its
Applications,  {136}, Cambridge University Press, Cambridge, 2010.


\bibitem{Lin} L. Li and C.L. Tang,
{Existence and multiplicity of solutions for a class of $p(x)$-Biharmonic equations},
{Acta Mathematica Scientia}, {33} (2013), 155-170.

\bibitem{Li-Sq} S. Liu and M. Squassina, On the existence of solutions to a fourth-order quasilinear
resonant problem, Abstr. Appl. Anal., 7 (2002), 125--133.

\bibitem{Sob6} Y. Liu,  R. Davidson and P. Taylor,  Investigation of the
touch sensitivity of ER fluid based tactile display, {Proceedings of SPIE, Smart Structures and Materials: Smart
Structures and Integrated Systems}, {5764} (2005), 92--99.

\bibitem{Mo-Re} G. Molica Bisci and D. Repov\v{s}, Multiple solutions of $p$-biharmonic equations with Navier
boundary conditions, Complex Variables and Elliptic Equations, 2014
59(2), 271--284.


\bibitem{M.T.G} T.G. Myers,
\emph{Thin films with high surface tension},
SIAM Review, 40 (3) (1998), 441--462.

\bibitem{PRad} P. Pucci and V. R\u adulescu, The impact of the mountain pass theory in nonlinear analysis: a mathematical survey,  {Boll. Unione Mat. Ital. Series IX},  {3}
(2010),  543--584.

\bibitem{rajruz} K.R. Rajagopal and M. R\.{u}\v{z}i\v{c}ka, Mathematical modelling of electrorheological fluids,
{Continuum Mech. Thermodyn.}, {13} (2001), 59--78.


\bibitem{RadRep} V. R\u adulescu and D. Repov\v{s}, {Partial Differential Equations
with Variable Exponents: Variational Methods and Quantitative
Analysis}, CRC Press, Taylor \& Francis Group, Boca Raton FL,
2015.

\bibitem{repovs1} D. Repov\v s, Stationary waves of Schr\"odinger-type equations with variable exponent, Anal. Appl. (Singap.) 13 (2015), no. 6, 645--661.

\bibitem{Ru} M. R\.{u}\v{z}i\v{c}ka,
{Electrorheological Fluids: Modeling and Mathematical Theory},
{Lecture Notes in Mathematics 1748, Springer-Verlag, Berlin},  2000.


\bibitem{Sob5} R. Stanway, J.L. Sproston and A.K. El-Wahed, Applications of electrorheological fluids in vibration control:
a survey, {Smart Mater. Struct.}, {5} (1996), 464--482.

\bibitem{ana2} Z. Y\"ucedag, Solutions of nonlinear problems involving $p(x)$-Laplacian operator, Adv. Nonlinear Anal. 4 (2015), no. 4, 285--293.

\bibitem{Will} M. Willem, Minimax Theorems, Birkh\"auser, Boston, 1996.

\bibitem{YL} H. Yin and Y. Liu, Existence of three solutions for a Navier boundary value problem involving the $p(x)$-biharmonic, Bull. Korean Math. Soc. 50 (2013),  1817--1826.

\bibitem{Fu} A. Zang, Y. Fu,
{Interpolation inequalities for derivatives in variable exponent Lebesgue-Sobolev
spaces},
{Nonlinear Anal. T.M.A.}, {69} (2008), 3629--3636.


\bibitem{Zhikov1} V.V. Zhikov,
{Averaging of functionals of the calculus of variations and elasticity theory},
{Math. USSR. Izv.}, f{9} (1987), 33--66.

\bibitem{Zhikov_dens} V.V. Zhikov, On the density of smooth functions in Sobolev-Orlicz spaces, Zap. Nauchn.
Sem. S.-Peterburg. Otdel. Mat. Inst. Steklov. (POMI), 310 (Kraev. Zadachi Mat. Fiz. i Smezh.
Vopr. Teor. Funkts. 35 [34]), 226 (2004), 67--81.

}
\end{thebibliography}
\end{document}